\documentclass[12pt]{amsart}
\usepackage{amsmath,amsfonts,amssymb}
\usepackage[all,cmtip]{xy}           
\usepackage{bm,bbm,bbding,txfonts,amscd}
\usepackage[all]{xy}
\usepackage[shortlabels]{enumitem}
\usepackage{ifpdf}
\ifpdf
\usepackage[colorlinks,final,backref=page,hyperindex]{hyperref}
\else
\usepackage[colorlinks,final,backref=page,hyperindex,hypertex]{hyperref}
\fi
\usepackage{tikz}
\usetikzlibrary{positioning}
\usepackage[active]{srcltx}
\usepackage{caption}

\makeatletter

\newtheorem{thm}{Theorem}[section]
\newtheorem{lem}[thm]{Lemma}
\newtheorem{cor}[thm]{Corollary}
\newtheorem{pro}[thm]{Proposition}
\newtheorem{ex}[thm]{Example}
\newtheorem{rmk}[thm]{Remark}
\newtheorem{defi}[thm]{Definition}

\newcommand{\be }{\begin{equation}}
\newcommand{\g}{\mathfrak g}

\newcommand{\huaB}{\mathcal{B}}
\newcommand{\dM}{\mathrm{d}}

\newcommand{\Id}{{\rm{Id}}}

\newcommand{\id}{\mathrm{id}}
\newcommand{\Der}{\mathrm{Der}}
\newcommand{\Inn}{\mathrm{Inn}}
\newcommand{\Ad}{\mathrm{Ad}}
\newcommand{\Aut}{\mathrm{Aut}}
\newcommand{\gl}{\mathfrak {gl}}
\newcommand{\ad}{\mathrm{ad}}
\newcommand{\Imm}{\mathrm{Im}}
\newcommand{\K}{\mathbb{K}}
\newcommand{\rb}{Rota-Baxter\,}
\newcommand{\rk}{\triangleright}

\begin{document}

\title{Inner post-Lie algebras and inner post-groups}

\author{Vsevolod Gubarev}
\address{Sobolev Institute of Mathematics Acad. Koptyug ave. 4, 630090 Novosibirsk, Russia}
\email{wsewolod89@gmail.com}

\author{Yue Li}
\address{Department of Mathematics, Jilin University, Changchun 130012, Jilin, China}
\email{liyue25@mails.jlu.edu.cn}

\author{Yunhe Sheng}
\address{Department of Mathematics, Jilin University, Changchun 130012, Jilin, China}
\email{shengyh@jlu.edu.cn}

\author{You Wang}
\address{Department of Mathematics, Jilin University, Changchun 130012, Jilin, China}
\email{wangyou20@mails.jlu.edu.cn}

\begin{abstract}
In this paper, using extension theory and cohomological approach we introduce the notion of the obstruction class for an inner post-Lie algebra being induced by a Rota-Baxter operator, and show that an inner post-Lie algebra is induced by a Rota-Baxter operator if and only if the obstruction class is trivial. Similarly, we introduce the notion of the obstruction class for an inner post-group being induced by a Rota-Baxter operator, and prove a parallel result. Finally, we give some applications of inner post-Lie algebras and inner post-groups.

\end{abstract}

\renewcommand{\thefootnote}{}
\footnotetext{2020 Mathematics Subject Classification.
17B38, 
17B56, 
}

\keywords{Rota-Baxter operators, post-Lie algebras, post-groups, cohomology, central extension}

\maketitle

\tableofcontents

\allowdisplaybreaks

\section{Introduction}

Post-Lie algebras were introduced by  Vallette in 2007~\cite{Vallette2007}
as algebras governing the Koszul dual operad to the operad of so-called commutative trialgebras.
A post-Lie algebra $(L,[\cdot,\cdot],\triangleright)$ is a Lie algebra~$(\g,[\cdot,\cdot])$ endowed with an additional bilinear operation~$\triangleright$ satisfying two identities on~$\g$.
In 2012, Burde, Dekimpe and Vercammen rediscovered post-Lie algebras while studying the analogue of the Milnor's question formulated for nil-affine transformations~\cite{Burde-2}. In 2013, Munthe-Kaas and Lundervold found that post-Lie algebras play important roles in numerical integration on  manifolds \cite{Munthe-Kaas-Lundervold}.  Meanwhile, it was found that post-Lie algebras also have important applications in regularity structures in stochastic analysis \cite{BK}. Due to these applications, post-Lie algebras are widely studied recently.

Post-Lie algebras are closely related to Rota-Baxter Lie algebras. 
Rota-Baxter operators on Lie algebras emerged in Semenov-Tian-Shansky's works as the operator form of the classical Yang-Baxter equation \cite{STS}.
In 2010, Bai, Guo and Ni showed how one can get a post-Lie algebra by defining a~new operation $x\triangleright y = [R(x),y]$ on a given Lie algebra endowed with a~Rota-Baxter operator~$R$ of weight~1 ~\cite{BGN}.
In~\cite{GK13}, Kolesnikov and the first author proved that any post-Lie algebra~$\g$ injectively embeds into Rota-Baxter Lie algebra defined on the space~$\g\oplus \g$.
In~\cite{UniEnv}, the first author constructed the universal enveloping Rota-Baxter Lie algebra of a given post-Lie algebra and showed that the analogue of the Poincar\'{e}-Birkhoff-Witt theorem holds.
In~\cite{BGN}, it was proved that any post-Lie algebra structure defined on a complete Lie algebra is induced via a~Rota-Baxter operator.
However, in general, not any post-Lie algebra originates from a~suitable Rota-Baxter operator as well as not every Lie algebra~$L$ is the commutator algebra of some associative product defined on the space~$L$.

To classify post-Lie algebras, the notion of inner post-Lie algebras were defined in~\cite{BM} for commutative post-Lie algebras and in~\cite{Burde2019} in general.
Originally, a post-Lie algebra $(L,[\cdot,\cdot],\triangleright)$ is called {\em inner} if there exists a linear operator $\varphi$ on $L$ such that $x\triangleright y = [\varphi(x),y]$ for all $x,y \in L$. The first  purpose of this paper is to clarify when an inner post-Lie algebra is induced by a~Rota-Baxter operator using extension theory and cohomological approach. More precisely, we associate a central extension, therefore a cohomological class after choosing a splitting, to an inner post-Lie algebra. Meanwhile, we characterize the linear map $\varphi$ in the above inner post-Lie algebra as a splitting of the central extension. The cohomological class obtained above is defined to be the {\em obstruction class} of the inner post-Lie algebra being induced by a Rota-Baxter operator. 
We prove that a~post-Lie algebra is induced by a~Rota-Baxter operator if and only if the obstruction class is trivial (Theorem~\ref{kap-RB-operator}).

All the main notions mentioned above have a group analogue. Rota-Baxter operators on groups were defined in 2021 by Guo, Lang and the third author in~\cite{GLS}. The differentiation of a~Rota-Baxter operator on a Lie group gives rise to a Rota-Baxter operator on the corresponding Lie algebra. Later in 2024, the notion of post-groups was introduced in \cite{BGST}, and the differentiation of a post-Lie group gives rise to a post-Lie algebra. A Rota-Baxter operator naturally gives rise to a post-group. All these structures are closely related to skew left braces, which were introduced in 2017 by Guarneri and Vendramin~\cite{GV2017} as a tool to generate set-theoretical solutions to the quantum Yang-Baxter equation. A skew left brace is a non-empty set~$G$ equipped with two group operations~$\cdot$ and $\circ$ such that a compatibility is satisfied.
In 2022, Bardakov and the first author showed~\cite{RBBraces} that given a Rota-Baxter group $(G,\cdot,B)$, one can obtain a skew left brace on $G$ as follows: $g\circ h = gB(g)hB(g)^{-1}$. See \cite{RS} for recent developments.
In 2023, Caranti and Stefanello studied the following question~\cite{Caranti}:
If a skew left brace $(G,\cdot,\circ)$ satisfies the condition
$g\circ h = g\varphi(g)h\varphi(g)^{-1}$ for some $\varphi\colon G\to G$,
under what conditions this brace is induced by some Rota-Baxter operator defined on the group $(G,\cdot)$? They found a~convenient criterion in terms of the trivialization of certain 2-cocycle   defined via $\varphi$~\cite[Theorem\,3.2]{Caranti}.

The second purpose of this paper is to clarify when an inner post-group is induced by a~Rota-Baxter operator using extension theory and cohomological approach similar as the case of inner post-Lie algebras. We introduce the notion of the obstruction class of an inner post-group being induced by a Rota-Baxter operator, and show that an inner post-group is induced by a~Rota-Baxter operator if and only if the obstruction class is trivial.
Although these results are intrinsically equivalent to those obtained in~\cite{Caranti}, we place more emphasis here on the use of extension theory and cohomological methods.
Moreover, we apply our approach to connect results about inner post-Lie algebras and inner post-groups on the level of the correspondence between Lie groups and Lie algebras.
We show that the differentiation of  an inner post-Lie group whose obstruction class is trivial  is an inner post-Lie algebra whose obstruction class is also trivial (Proposition~\ref{pro:obs-post-group}).

Finally, we consider applications of inner post-Lie algebras and inner post-Lie groups.
In 2022,   Burde,   Dekimpe  and   Monadjem proved~\cite{Burde2022} that given a finite-dimensional post-Lie algebra $(\g,[\cdot,\cdot]_\g,\triangleright)$ over $\mathbb{C}$ such that $\g_\triangleright$ is semisimple we have that $\g$ is also semisimple and
$\g\cong \g_\triangleright$. This deep and significant result involves several non-trivial steps. One of them is a~consequence of the results of  Onishchik~\cite{Onishchik62,Onishchik69} concerning decompositions of semisimple Lie algebras. Involving the latter, we provide an alternative proof of the Burde-Dekimpe-Monadjem theorem in the case when a post-Lie algebra is inner (Theorem~\ref{thm-2}).
Finally,  under certain mild conditions, we show that the Lie groups of the sequence of the sub-adjacent Lie groups $(G_n,\circ_n)$ associated to an inner post-group whose obstruction class is trivial, are semisimple for all $n\geq 0$.
(Theorem~\ref{thm-3}).

\vspace{2mm}
\noindent
{\bf Acknowledgements}.
This research is supported by NSF of Jilin Province (20260101013JJ), NSFC (12471060, W2412041).
The first author was supported by the Russian Science Foundation (project 25-41-00005).

\section{A cohomological characterization of inner post-Lie algebras}

In this section, using the theory of central extensions of Lie algebras, we introduce the obstruction class of an inner post-Lie algebra being induced from a Rota-Baxter operator. 
We show that an inner post-Lie algebra is induced from a Rota-Baxter operator if and only if the  obstruction class is trivial.

\begin{defi}
A \textbf{post-Lie algebra} $(\g,[\cdot,\cdot]_\g,\triangleright)$ consists of a Lie algebra $(\g,[\cdot,\cdot]_\g)$ and a binary product $\triangleright:\g\otimes\g\to \g$ such that
\begin{eqnarray}
\label{Post-1} x\triangleright [y,z]_\g&=&[x\triangleright y,z]_\g+[y,x\triangleright z]_\g,\\
\label{Post-2} ([x,y]_\g+x\triangleright y-y\triangleright x) \triangleright z&=&x\triangleright(y\triangleright z)-y\triangleright(x\triangleright z),\quad \forall x, y, z\in \g.
\end{eqnarray}
\end{defi}

\begin{rmk}
Let $(\g,[\cdot,\cdot]_\g,\triangleright)$ be a post-Lie algebra. If the Lie bracket $[\cdot,\cdot]_\g$ vanishes, then $(\g,\triangleright)$ becomes a pre-Lie algebra. Thus, a post-Lie algebra can be viewed as a nonabelian generalization of a pre-Lie algebra \cite{Bai,Burde-1,Ma}.
\end{rmk}

The following characterization of post-Lie algebras is given in \cite{BGN}.

\begin{lem}{\rm(\cite{BGN})}\label{post-Lie-sub}
Any post-Lie algebra $(\g,[\cdot,\cdot]_\g,\triangleright)$ gives rise to a new Lie algebra $(\g,[\cdot,\cdot]_{\triangleright})$
given by
$$
[x,y]_{\triangleright} := x\triangleright y-y\triangleright x+[x,y]_\g,\quad\forall x,y\in\g.
$$
which is called the \textbf{sub-adjacent Lie algebra} of the post-Lie algebra $(\g,[\cdot,\cdot]_\g,\triangleright)$ and denoted by~$\g_\triangleright$. Moreover, the left multiplication  $L_{\triangleright}:\g\to \Der(\g)$ defined by $L_{\triangleright}(x)(y)=x\triangleright y$ is an action of the sub-adjacent Lie algebra $\g_\triangleright$ on the Lie algebra $(\g,[\cdot,\cdot]_\g)$.
\end{lem}

Denote by $Z(\g)$ the center of $(\g,[\cdot,\cdot]_\g)$.
\begin{lem}
Let $ (\g,[\cdot,\cdot]_\g,\triangleright)$ be a post-Lie algebra. Then $ L_\rk$ induces a representation of $\g_\rk$ on~$Z(\g)$.
\end{lem}	
\begin{proof}
By Eq. \eqref{Post-1}, for any $y\in Z(\g), x, z\in\g$, we have
$$
[{L_\rk}(x)(y),z]_\g=x\rk [y,z]_\g-[y,x\rk z]_\g=0,
$$
which implies that ${L_\rk}(x)(y)\in Z(\g)$. Thus, $ L_\rk$ induces a representation of $\g_\rk$ on $Z(\g)$.
\end{proof}

Note that for a post-Lie algebra $(\g,[\cdot,\cdot]_\g,\triangleright)$, the left multiplication  $L_{\triangleright}(x)$ takes values in $\Der(\g)$ for all $x\in \g$.
Now we introduce the notion of inner post-Lie algebras.
Denote by $\Inn(\g)$ the inner derivation Lie algebra of a Lie algebra $(\g,[\cdot,\cdot]_\g)$.

\begin{defi}{\rm (\cite{Burde2019,BM})}
A post-Lie algebra $(\g,[\cdot,\cdot]_\g,\triangleright)$ is called \textbf{inner},
if $\Imm(L_{\triangleright}) \subseteq \Inn(\g).$
\end{defi}

A class of inner post-Lie algebras is given by Rota-Baxter operators on Lie algebras.

\begin{defi}
Let $(\g,[\cdot,\cdot]_\g)$ be a Lie algebra. A linear operator $R:\g \to \g$ is called a \textbf{Rota-Baxter operator} (of weight $1$) if it satisfies the relation:
\begin{eqnarray}\label{RBO-Lie-alg}
[R(x),R(y)]_\g=R([R(x),y]_\g+[x,R(y)]_\g+[x,y]_{\g}),\quad \forall x,y\in \g.
\end{eqnarray}
The triple $(\g,[\cdot,\cdot]_\g,R)$ is called a \textbf{Rota-Baxter Lie algebra}.
\end{defi}

Rota-Baxter Lie algebras naturally induce post-Lie algebras, which are inner.

\begin{lem}\label{lem-RB-innPL}{\rm(\cite{BGN})}
Let $(\g,[\cdot,\cdot]_\g,R)$ be a Rota-Baxter Lie algebra. Then $(\g,[\cdot,\cdot]_\g,\triangleright_{R})$ is a post-Lie algebra, where the multiplication $\triangleright_{R}$ is given by
\begin{eqnarray}\label{RBO-post-Lie-alg}
x \triangleright_{R} y:= [R(x),y]_{\g}, \quad \forall x,y\in \g.
\end{eqnarray}
Moreover, this post-Lie algebra is inner.
\end{lem}

\begin{cor}
Let $(\g,[\cdot,\cdot]_\g,R)$ be a Rota-Baxter Lie algebra, and $(\g,[\cdot,\cdot]_\g,\triangleright_{R})$ the induced post-Lie algebra. Then the induced representation $ L_\rk$  of $\g_\rk$ on $Z(\g)$ is trivial.
\end{cor}
\begin{proof}
It follows from $L_\rk(x)(y)=[R(x),y]_\g=0$ for all $x\in\g$ and $y\in Z(\g)$.
\end{proof}

\begin{pro}
Two Rota-Baxter operators $R_1, R_2$ yield the same post-Lie algebra $(\g,[\cdot,\cdot]_\g,\triangleright)$ by Eq. \eqref{RBO-post-Lie-alg} if and only if there exists a $1$-cocycle $t \colon \g_\triangleright \to Z(\g)$ on the Lie algebra $\g_\triangleright$ with the coefficients in the trivial representation $Z(\g)$ such that $$R_2(x) = R_1(x) + t(x),\quad \forall x \in \g.$$
\end{pro}

\begin{proof}
On the one hand, suppose that two Rota-Baxter operators $R_1, R_2$ yield the same post-Lie algebra $(\g,[\cdot,\cdot]_\g,\triangleright)$ by Eq. \eqref{RBO-post-Lie-alg}. Then $x\triangleright y  =[R_1(x),y]_{\g} = [R_2(x),y]_{\g}$ for all $x,y\in \g$. Define $t(x):= R_2(x)-R_1(x)$ for all $x \in \g$, then we have $[t(x),y]_{\g}=[R_2(x)-R_1(x),y]=0$, which implies that
$\Imm(t)\subseteq Z(\g)$.
Further, by Eq. \eqref{RBO-Lie-alg}, we have
\begin{eqnarray*}
t([x,y]_{\triangleright})
&=& R_2([x,y]_{\triangleright}) - R_1([x,y]_{\triangleright})\\
&=& [R_2(x),R_2(y)]_{\g} - [R_1(x),R_1(y)]_{\g} \\
&=& [R_1(x),t(y)]_\g +[t(x),R_1(y)]_\g +[t(x),t(y)]_\g\\
&=& 0,
\end{eqnarray*}
which implies that $t([x,y]_{\triangleright})=0=x\rk t(y)-y\rk t(x)$ for all $x,y\in \g$. Thus, $t$ is a $1$-cocycle on the Lie algebra $\g_\rk$ with the coefficients in the trivial representation $Z(\g)$.

On the other hand, suppose that there exists a $1$-cocycle $t \colon \g_\triangleright \to Z(\g)$ on the Lie algebra~$\g_\triangleright$ with the coefficients in the trivial representation $Z(\g)$ such that $R_2(x) = R_1(x) + t(x)$ for all $x \in \g$. Then $R_1$ and $R_2$ yield the same post-Lie algebra $(\g,[\cdot,\cdot]_\g,\triangleright)$, since $x\triangleright y =[R_1(x),y]_{\g} =[R_2(x)-t(x),y]_{\g}= [R_2(x),y]_{\g}$.
\end{proof}

A natural question arises:

\begin{center}
{\bf When is an inner post-Lie algebra induced by a Rota-Baxter operator?}
\end{center}

In the sequel, we use the cohomological approach to solve this problem.
Consider the adjoint representation $\ad:\g \to \Der(\g)$ as a homomorphism of Lie algebras. Since $\ker(\ad)=Z(\g)$, we have $\g/ Z(\g) \cong \Imm(\ad)=\Inn(\g)$ by the First Isomorphism Theorem in the case of Lie algebras. Then we have the following central extension of Lie algebras:
\[
\xymatrix{
0 \ar[r] & Z(\g)  \ar[r]^{i } & \g \ar[r]^{ p\quad\quad} & \g/ Z(\g) \cong \Inn(\g)\ar[r] & 0,
}\]
where $p$ is the canonical projection and $i$ is the inclusion. For simplicity, we will

Assume that $(\g,[\cdot,\cdot]_\g,\triangleright)$ is an inner post-Lie algebra. Then by pullback along $ L_{\triangleright}: \g_{\triangleright} \to \Inn(\g)$, we obtain the following commutative diagram, which consists of two central extensions of Lie algebras:
\begin{eqnarray}\label{diagram1}
\xymatrix{
0 \ar[r] & Z(\g)  \ar[d]_{\Id}   \ar[r]^{i_2} & \g^{(2)} \ar[d]_{p_2}  \ar[r]^{p_1} &\g_\triangleright      \ar[d]_{L_{\triangleright}}   \ar[r] &0\\
0 \ar[r] & Z(\g)    \ar[r]^{i} & \g   \ar[r]^{\ad} & \Inn(\g)           \ar[r] &0,
}
\end{eqnarray}
where $\g^{(2)}$ is a Lie subalgebra of $\g_{\triangleright}\oplus \g$, which is the direct product of two Lie algebras
$\g_\triangleright$ and $(\g,[\cdot,\cdot]_\g)$, given by
\begin{eqnarray}\label{g-2}
\g^{(2)}=\{(x,y)\in \g_\triangleright \oplus \g~|~L_{\triangleright}(x)= \ad_{y}\},
\end{eqnarray}
$i$ is the inclusion, $i_2$ is the inclusion to the second summand and $p_1,p_2$ are projections to the first summand and the second summand respectively. According to Eq. \eqref{g-2}, it is obvious that the right square commutes.

Now choose a section of the central extension of Lie algebras in the first line of the diagram~\eqref{diagram1}, which is a linear map $s:\g_\triangleright  \to \g^{(2)}$ such that $p_1\circ s=\Id_{\g_\triangleright}$. Obviously, $s$ is of the form $s(x)=(x,\phi(x))$ for a linear map $\phi:\g \to \g$. Since the right square commutes, we have
$$ \ad \circ p_2 \circ s(x)=L_{\triangleright} \circ p_1 \circ s(x)= L_{\triangleright}(x), $$
which implies $L_{\triangleright}(x)= \ad_{\phi(x)}$.

Define $\kappa:\wedge^2\g_\triangleright\to \g^{(2)}$ by
$$
\kappa(x,y)=[s(x),s(y)]^{(2)}-s([x,y]_\triangleright),\quad \forall x,y
\in \g_\triangleright.
$$
Then it is obvious that $p_1(\kappa(x,y))=0,$ which implies that $\kappa(x,y)\in i_2(Z(\g))$. Identifying $Z(\g)$ and $i_2(Z(\g))$, it follows that $\kappa:\wedge^2\g_\triangleright\to Z(\g)$ is given by
\begin{eqnarray}\label{kappa}
\kappa(x, y)=[\phi(x),\phi(y)]_{\g}-\phi([x, y]_{\triangleright}).
\end{eqnarray}

On the other hand, the induced representation $\rho_s:\g_\rk\to \gl(Z(\g))$ of the section $s$ is given by
$$
\rho_s(x)(y)=[(x,\phi(x)),(0,y)]^{(2)}=[\phi(x),y]_\g=0,\quad \forall x\in\g, y\in Z(\g),
$$
which is trivial.

\begin{pro}\label{kap-2-cocycle}
Let $(\g,[\cdot,\cdot]_\g,\triangleright)$ be an inner post-Lie algebra, $s$ and $\kappa$ defined as above.
Then $\kappa$ is a $2$-cocycle on the Lie algebra $\g_\triangleright$ with the coefficients in the trivial representation $Z(\g)$. Moreover, its cohomology class in $H^2(\g_\triangleright, Z(\g))$ does not depend on the choice of the section $s$.
\end{pro}

\begin{proof}
In fact, this conclusion follows from the general theory of central extensions of Lie algebras. To be self-contained, we roughly give the proof.

By the Jacobi identity and by Eq. \eqref{kappa}, for all $x,y,z\in\g$, we have
\begin{eqnarray*}
&&\kappa([x,y]_{\triangleright},z)+\kappa([y,z]_{\triangleright},x)+\kappa([z,x]_{\triangleright},y)\\
&=& [\phi([x,y]_{\triangleright}),\phi(z)]_{\g}-\phi([[x,y]_{\triangleright},z]_{\triangleright})
+[\phi([y,z]_{\triangleright}),\phi(x)]_{\g}-\phi([[y,z]_{\triangleright},x]_{\triangleright})\\
&&+[\phi([z,x]_{\triangleright}),\phi(y)]_{\g}-\phi([[z,x]_{\triangleright},y]_{\triangleright})\\
&=&  [\phi([x,y]_{\triangleright}),\phi(z)]_{\g}+[\phi([y,z]_{\triangleright}),\phi(x)]_{\g}
+[\phi([z,x]_{\triangleright}),\phi(y)]_{\g}\\
&=&[[\phi(x),\phi(y)]_{\g},\phi(z)]_{\g}
+[[\phi(y),\phi(z)]_{\g},\phi(x)]_{\g}+[[\phi(z),\phi(x)]_{\g},\phi(y)]_{\g}\\
&&-[\kappa(x,y),\phi(z)]_\g-[\kappa(y,z),\phi(x)]_\g-[\kappa(z,x),\phi(y)]_\g\\
&=&0,
\end{eqnarray*}
which implies that $\kappa$ is a $2$-cocycle on the Lie algebra $\g_\triangleright$ with the coefficients in the trivial representation $Z(\g)$.

Suppose that there exists another section $s':\g_\triangleright \to \g^{(2)}$ given by $s'(x)=(x,\phi'(x))$ for a linear map $\phi':\g \to \g$. Since the right square commutes in diagram \eqref{diagram1}, we have
$$ \ad_{\phi'(x)}=\ad \circ p_2 \circ s'(x)=L_{\triangleright} \circ p_1 \circ s'(x)= L_{\triangleright}(x),\quad \forall x\in \g. $$
Define $t=\phi'-\phi:\g \to \g$; then we have
$$[t(x),y]_{\g}=[\phi'(x),y]_{\g} - [\phi(x),y]_{\g}=x \triangleright y-x \triangleright y=0, \quad \forall x,y\in \g,$$
which implies that $\Imm (t)\subseteq Z(\g)$.

Next we define $\kappa' \colon \wedge^2\g_\triangleright \to Z(\g)$ by
\begin{eqnarray*}
\kappa'(x, y)=[\phi'(x),\phi'(y)]_{\g}-\phi'([x, y]_{\triangleright}),\quad \forall x,y\in \g_\rk.
\end{eqnarray*}
Then we have
\begin{eqnarray*}
\kappa'(x, y)&=&[\phi'(x),\phi'(y)]_{\g}-\phi'([x, y]_{\triangleright})\\
&=& [(\phi+t)(x),(\phi+t)(y)]_{\g}-(\phi+t)([x, y]_{\triangleright})\\
&=& [\phi(x),\phi(y)]_{\g}-\phi([x, y]_{\triangleright})+[\phi(x),t(y)]_{\g}+[t(x),\phi(y)]_{\g}\\
&&+[t(x),t(y)]_{\g}-t([x, y]_{\triangleright})\\
&=& \kappa(x, y)-t([x, y]_{\triangleright})\\
&=& \kappa(x, y)+\dM t(x,y),
\end{eqnarray*}
which implies that $[\kappa']=[\kappa]$ in the second cohomology group $H^2(\g_\triangleright, Z(\g))$. Thus, the cohomological class $[\kappa]$ does not depend on the choice of $s$.
\end{proof}

\begin{defi}
The cohomological class $[\kappa]\in H^2(\g_\triangleright, Z(\g))$ is called the {\bf obstruction class} of the inner post-Lie algebra $(\g,[\cdot,\cdot]_\g,\triangleright)$ being induced by a Rota-Baxter operator.
\end{defi}

\begin{thm}\label{kap-RB-operator}
An inner post-Lie algebra  $(\g,[\cdot,\cdot]_\g,\triangleright)$ is induced by a Rota-Baxter operator if and only if the obstruction class $[\kappa]$ in $H^2(\g_\triangleright, Z(\g))$ is trivial.	
\end{thm}

\begin{proof}
If $\kappa=0$, for all $x,y\in \g$, we have
\begin{eqnarray*}
[\phi(x),\phi(y)]_{\g}&=&\phi([x, y]_{\triangleright})= \phi([\phi(x),y]_{\g}+[x,\phi(y)]_{\g}+[x,y]_{\g}),
\end{eqnarray*}
which implies that $\phi$ is a Rota-Baxter operator of weight $1$ on the Lie algebra $(\g,[\cdot,\cdot]_{\g})$.

Suppose that the obstruction class $[\kappa]$ in $H^2(\g_\triangleright, Z(\g))$ is trivial, i.\,e. $[\kappa]=[0]$, there exists a~linear map $t:\g \to Z(\g)$ such that $\kappa(x, y)=\dM t(x,y)=-t([x, y]_{\triangleright})$.
Define $R = \phi - t$. Let us show that $R$~is a Rota-Baxter operator of weight $1$ on $(\g,[\cdot,\cdot]_{\g})$:
\begin{eqnarray*}
[R(x),R(y)]_{\g}-R([x,y]_{\triangleright})
&=& [\phi(x)-t(x),\phi(y)-t(y)]_{\g}-(\phi-t)([x,y]_{\triangleright})\\
&=& [\phi(x),\phi(y)]_{\g}-[\phi(x),t(y)]_\g-[t(x),\phi(y)]_\g+[t(x),t(y)]_\g\\
&&-\phi([x,y]_{\triangleright})+t([x,y]_{\triangleright})\\
&=& [\phi(x),\phi(y)]_{\g}-\phi([x,y]_{\triangleright})+t([x,y]_{\triangleright})\\
&=& \kappa(x, y)+t([x, y]_{\triangleright})\\
&=& 0,
\end{eqnarray*}
since $\Imm(t)\subseteq Z(\g)$. The other direction of the proof is obvious, so we omit details.
\end{proof}

Recall that a \textbf{complete Lie algebra} is a Lie algebra such that its center is zero and all its derivations are inner.

\begin{cor}{\rm(\cite[Remark 5.11]{BGN})}. \label{cor-complete-Lie}
For any post-Lie algebra $(\g,[\cdot,\cdot]_\g,\triangleright)$, if the Lie algebra $(\g,[\cdot,\cdot]_\g)$ is complete, then the post-Lie algebra $(\g,[\cdot,\cdot]_\g,\triangleright)$ is inner and is defined via a Rota-Baxter operator  $\phi$ of weight $1$.
\end{cor}

\begin{proof}
As it was noted in \cite{BGN}, since the Lie algebra $(\g,[\cdot,\cdot]_\g)$ is complete, all its derivations are inner. Hence there exists a map $\phi:\g \to \g$ such that $x \triangleright y=[\phi(x),y]_{\g}$, which implies that the post-Lie algebra $(\g,[\cdot,\cdot]_\g,\triangleright)$ is inner.
Moreover, since $Z(\g)$ is trivial, due to the definition of $\kappa$ we have
$\kappa(x, y)=[\phi(x),\phi(y)]_{\g}-\phi([x, y]_{\triangleright})=0$. Then by Theorem \ref{kap-RB-operator}, $\phi$ is a Rota-Baxter operator of weight $1$.
\end{proof}

\begin{ex}{\rm(\cite{LBG})}.
{\rm Let $(\g=\mathfrak{sl}(2,\mathbb{C}),[\cdot,\cdot]_{\g})$ be a 3-dimensional complex simple Lie algebra with a basis:
\begin{eqnarray*}
\begin{Bmatrix}
{e_1=\frac{1}{2}\begin{pmatrix}
            0 & 1 \\
            -1 & 0
\end{pmatrix}}, &
{e_2= \frac{1}{2i}\begin{pmatrix}
            0 & 1 \\
            1 & 0
\end{pmatrix}},&
{e_3=\frac{1}{2i}\begin{pmatrix}
            1 & 0 \\
            0 & -1
\end{pmatrix}}			
\end{Bmatrix},\quad i=\sqrt{-1},
\end{eqnarray*}
whose Lie bracket satisfies $[e_1,e_2]_\g=e_3$, $[e_2,e_3]_\g=e_1$ and  $[e_3,e_1]_\g=e_2$. Define $\triangleright$ on $(\g,[\cdot,\cdot]_\g)$ by
\begin{eqnarray*}
e_1\triangleright e_1&=0,\qquad\qquad\quad e_1\triangleright e_2=e_3, \quad\;\;\quad e_1\triangleright e_3=-e_2,\\
e_2\triangleright e_1&=\frac{i}{2}e_2+\frac{1}{2}e_3,\quad\;\; e_2\triangleright e_2=-\frac{i}{2}e_1,\quad e_2\triangleright e_3=-\frac{1}{2}e_1,\\
e_3\triangleright e_1&=-\frac{1}{2}e_2+\frac{i}{2}e_3,\quad e_3\triangleright e_2=\frac{1}{2}e_1,\quad\;\; e_3\triangleright e_3=-\frac{i}{2}e_1.
\end{eqnarray*}
Then $(\g,[\cdot,\cdot]_{\g},\triangleright)$ is a post-Lie algebra. Define a linear map $P: \g\rightarrow \g$  by
\begin{eqnarray*}
P(e_1)=e_1,\quad P(e_2)=-\frac{1}{2}e_2+\frac{i}{2}e_3,\quad P(e_3)=-\frac{i}{2}e_2-\frac{1}{2}e_3.
\end{eqnarray*}
It is easy to see that $P$ is a Rota-Baxter operator of weight $1$ and
Eq.~\eqref{RBO-post-Lie-alg} holds. Therefore, $(\g,[\cdot,\cdot]_\g,\triangleright)$ is an inner post-Lie algebra, whose obstruction class is trivial.	
}\end{ex}

\begin{ex}
{\rm
Let  $\g=\K e_1\oplus \K e_2\oplus \K e_3$ with nonzero product $[e_1,e_2]_\g=e_2$. Then for any $\alpha,\beta,\gamma\in \K$,
\begin{eqnarray*}
\varphi=\begin{pmatrix}
    1 & 0 & 0 \\
    0 & -1 & 0 \\
    \alpha & \beta & \gamma
\end{pmatrix}
\end{eqnarray*}
defines an inner post-Lie algebra $(\g,[\cdot,\cdot]_\g,\triangleright)$ by $x\triangleright y=[\varphi (x),y]_\g$ for all $x,y\in \g$. By {\rm \cite{Burde2019}}, $\varphi$ is a~Rota-Baxter operator of weight $1$ on $(\g,[\cdot,\cdot]_{\g})$ if and only if $\beta=0$.

Next, suppose  $\beta \neq 0$. For arbitrary elements $x = k_1 e_1 + k_2 e_2 + k_3 e_3$ and $y = l_1 e_1 + l_2 e_2 + l_3 e_3$ in $\mathfrak{g}$, define $\kappa$ by Eq.~\eqref{kappa}. Then we have $$
\kappa(x,y)=[\varphi(x),\varphi(y)]_{\g}-\varphi([x,y]_{\triangleright})=\beta(k_2l_1-k_1l_2)e_3.$$ By Proposition \ref{kap-2-cocycle}, $\kappa$ is a $2$-cocycle on the sub-adjacent Lie algebra $\g_\triangleright$ with the coefficients in the trivial representation $Z(\g)$. Furthermore,  for any $\lambda,\mu\in \K$,
\begin{eqnarray*}
t=\begin{pmatrix}
    0 & 0 & 0 \\
    0 & 0 & 0 \\
    \lambda & \beta & \mu
\end{pmatrix}
\end{eqnarray*}
defines a linear map from $\g$ to $Z(\g)$ such that $\kappa(x, y)=\dM t(x,y)=-t([x, y]_{\triangleright})$. Hence, the obstruction class $[\kappa]$ in $H^2(\g_\triangleright, Z(\g))$ is trivial.
Then by Theorem \ref{kap-RB-operator}, $R:=\varphi-t$ is a Rota-Baxter operator of weight $1$ on $(\g,[\cdot,\cdot]_{\g})$.
}\end{ex}

\section{A cohomological characterization of inner post-groups}

In this section, we introduce the obstruction class of an inner post-group being induced by a~Rota-Baxter operator. We show that an inner post-group is  induced by a Rota-Baxter operator if and only if the obstruction class is trivial.

Let us first recall the notion of post-groups \cite{BGST}.

\begin{defi}{\rm (\cite{BGST})}
A \textbf{post-group} $(G,\cdot,\rhd)$ is a group $(G,\cdot)$ equipped with a multiplication $\rhd:G \times G \to G$ such that
\begin{itemize}
\item[{\rm(i)}] for all $a\in G$, the left multiplication operator
$$ L^{\rhd}({a}):G \to G, \quad L^{\rhd}(a)(b)=a \rhd b, \quad \forall b\in G $$
is an automorphism of the group $(G,\cdot)$. That means,
\begin{equation}\label{post-group-1}
a \rhd (b\cdot c)=(a\rhd b) \cdot (a \rhd c), \quad \forall a,b,c\in G;
\end{equation}
\item[{\rm(ii)}] the following ``weighted" associativity holds,
\begin{equation}\label{post-group-2}
(a\cdot (a\rhd b))\rhd c=a \rhd (b\rhd c), \quad \forall a,b,c\in G.
\end{equation}
\end{itemize}
If the group $(G,\cdot)$ is an abelian group, then $(G,\cdot,\rhd)$ is called a \textbf{pre-group}.
\end{defi}

The following characterization of post-groups is given in \cite{BGST}.

\begin{lem}\label{post-group-sub}{\rm (\cite{BGST})}
Let $(G,\cdot,\rhd)$ be a post-group. Define $\circ:G \times G \to G$ by
$$
a \circ b=a \cdot (a \rhd b),\quad \forall a,b\in G.
$$
Then $(G,\circ)$ is a group, which is called the \textbf{sub-adjacent group} of the post-group $(G,\cdot,\rhd)$. Moreover, Eqs.~\eqref{post-group-1}-\eqref{post-group-2} equivalently mean that the left multiplication $L^{\rhd}:G \to \Aut(G)$ is an action of the sub-adjacent group $(G,\circ)$ on the group $(G,\cdot)$.
\end{lem}

Denote by $Z(G)$ the center of $(G,\cdot)$.
\begin{lem}
Let $(G,\cdot,\rhd)$ be a post-group. Then $L^{\rhd}$ induces a group action of $(G,\circ)$ on $Z(G)$.
\end{lem}

\begin{proof}
For any $a,c\in G$, since $L^{\rhd}(a)$ is an automorphism of the group $(G,\cdot)$, there exists unique $d\in G$ such that $a \rhd d=c$.
Then for any $b\in Z(G)$, by Eq. \eqref{post-group-1}, we have
\begin{eqnarray*}
 \Ad_{a \rhd b} (c)&=&\Ad_{a \rhd b} (a \rhd d)\\
 &=&(a \rhd b) \cdot (a \rhd d) \cdot (a \rhd b)^{-1}\\
 &=&(a \rhd b) \cdot (a \rhd d) \cdot (a \rhd b^{-1})\\
 &=&a \rhd (\Ad_b~d)\\
 &=&a \rhd d=c,
\end{eqnarray*}
which implies that $a \rhd b \in Z(G)$. Therefore, $L^{\rhd}$ induces a group action of $(G,\circ)$ on $Z(G)$.
\end{proof}

Note that for a post-group $(G,\cdot,\rhd)$, the left multiplication $L^{\rhd}(a)$ takes values in $\Aut(G)$ for all $a\in G$.
Now we introduce the notion of inner post-groups. Let $(G,\cdot)$ be a group. Denote by $\Inn(G)$ the inner automorphism group of $(G,\cdot)$.

\begin{defi}
A post-group $(G,\cdot,\rhd)$ is called \textbf{inner}, if $\Imm(L^{\rhd})\subseteq \Inn(G).$
\end{defi}

A class of inner post-groups are given by Rota-Baxter operators on groups.

\begin{defi}{\rm (\cite{GLS})}
Let $(G,\cdot)$ be a group. A map $\huaB:G \to G$ is called a \textbf{Rota-Baxter operator} if it satisfies the following equation:
\begin{equation}\label{RBO-group}
\huaB(a) \cdot \huaB(b)=\huaB( a \cdot \Ad_{\huaB(a)} b),\quad \forall a,b\in G.
\end{equation}
the triple	$(G,\cdot,\huaB)$ is called a \textbf{Rota-Baxter group}.
\end{defi}

Rota-Baxter groups naturally give rise to post-groups, which are inner.

\begin{lem}\label{lem:RB-group-post-group}{\rm (\cite{BGST,RBBraces})}
Let $(G,\cdot,\huaB)$ be a Rota-Baxter group. Then $(G,\cdot,\rhd_{\huaB})$ is a post-group, where the binary operation $\rhd_{\huaB}$ is given by
\begin{equation}\label{rhd-huaB}
a \rhd_{\huaB} b := \Ad_{\huaB(a)}b=  \huaB(a) \cdot b \cdot (\huaB(a))^{-1}, \quad \forall a,b\in G.
\end{equation}
Moreover, this post-group is inner.
\end{lem}

\begin{cor}
Let $(G,\cdot,\huaB)$ be a Rota-Baxter group, and $(G,\cdot,\rhd_{\huaB})$ the induced post-group. Then the induced group action $L^{\rhd}$ of
$(G,\circ)$ on $Z(G)$ is trivial.
\end{cor}

\begin{proof}
It follows from $L^{\rhd}(a)(b)= \Ad_{\huaB(a)}b=b$ for all $a \in G$ and $b \in Z(G)$.
\end{proof}

\begin{pro}
Two Rota-Baxter operators $\huaB_1, \huaB_2$ yield the same post-group $(G,\cdot,\rhd)$ by Eq.~\eqref{rhd-huaB} if and only if there exists a $1$-cocycle $\zeta \colon (G,\circ) \to Z(G)$ on the group $(G,\circ)$ with the coefficients in the trivial group action $Z(G)$ such that $\huaB_2(a)=\huaB_1(a)\cdot \zeta(a) $ for all $a \in G$.
\end{pro}

\begin{proof}
On the one hand, suppose that two Rota-Baxter operators $\huaB_1, \huaB_2$ yield the same post-group $(G,\cdot,\rhd)$ by Eq. \eqref{rhd-huaB}, it means that $a \rhd b=\Ad_{\huaB_1(a)} b=\Ad_{\huaB_2(a)} b $ for all $a,b\in G$. Define $\zeta(a):=(\huaB_1(a))^{-1}\cdot \huaB_2(a)$ for all $a\in G$, then we have
$$ \zeta(a)\cdot b=(\huaB_1(a))^{-1}\cdot \huaB_2(a)\cdot b=b \cdot (\huaB_1(a))^{-1}\cdot \huaB_2(a)=b \cdot \zeta(a),$$
which implies that $\Imm(\zeta)\subseteq Z(G)$. Moreover, we also have
\begin{eqnarray*}
\zeta(a \circ b)&=&\huaB_1(a \circ b)^{-1} \cdot \huaB_2(a \circ b)\\
&=& (\huaB_1(a) \cdot \huaB_1(b))^{-1} \cdot \huaB_2(a) \cdot \huaB_2(b)\\
&=& \huaB_1(b)^{-1} \cdot \huaB_1(a)^{-1} \cdot \huaB_2(a) \cdot \huaB_2(b)\\
&=& \huaB_1(b)^{-1} \cdot \zeta(a) \cdot \huaB_2(b)\\
&=&  \zeta(a) \cdot \huaB_1(b)^{-1}\cdot \huaB_2(b)\\
&=&  \zeta(a) \cdot \zeta(b),
\end{eqnarray*}
which implies that $\zeta(a \circ b)= \zeta(a) \cdot \zeta(b)=\zeta(a)\cdot (a\rk \zeta(b))$ for all $a,b\in G$. Thus,   $\zeta$ is a $1$-cocycle on the group $(G,\circ)$ with the coefficients in trivial group action $Z(G)$.

On the other hand, suppose that there exists a $1$-cocycle $\zeta \colon (G,\circ) \to Z(G)$ on the group $(G,\circ)$ with the coefficients in the trivial group action $Z(G)$ such that $\huaB_2(a)= \huaB_1(a)\cdot \zeta(a) $ for all $a \in G$. Then  $\huaB_1$ and $\huaB_2$ yield the same post-group $(G,\cdot,\rhd)$, since
$$a \rhd b=\Ad_{\huaB_2(a)} b=\Ad_{\huaB_1(a)\cdot \zeta(a)} b=\Ad_{\huaB_1(a)}b,\quad \forall       a,b\in G.$$
The proof is finished.
\end{proof}

A natural question arises:

\begin{center}
{\bf	When is an inner post-group induced by a Rota-Baxter operator?}
\end{center}

Consider the adjoint representation $\Ad:G \to \Aut(G)$ as a homomorphism of groups. Since $\ker(\Ad)=Z(G)$, we have $G/ Z(G) \cong \Imm(\Ad)=\Inn(G)$ by the First Isomorphism Theorem in the case of groups. Then we have the following central extension of groups:
\[
\xymatrix{
0 \ar[r] & Z(G)  \ar[r]^{i } & (G,\cdot) \ar[r]^{ p\quad\quad} & G/ Z(G) \cong \Inn(G)\ar[r] & 0.
}\]

Assume that $(G,\cdot,\rhd)$ is an inner post-group. Then by pullback along $L^{\rhd}:G \to \Aut(G)$, we obtain the following commutative diagram, which consists of two central extensions of groups:
\begin{eqnarray}\label{diagram2}
\xymatrix{
0 \ar[r] & Z(G)  \ar[d]_{\Id}   \ar[r]^{i_2} & G^{(2)} \ar[d]_{p_2}  \ar[r]^{p_1} &(G,\circ)     \ar[d]_{L^{\rhd}}   \ar[r] &0\\
0 \ar[r] & Z(G)    \ar[r]^{i} & (G,\cdot)   \ar[r]^{\Ad} & \Inn(G)           \ar[r] &0,
}
\end{eqnarray}
where $(G^{(2)},\bullet)$ is a subgroup of the direct product between two groups
$(G,\circ)$ and $(G,\cdot)$, given by
\begin{eqnarray}\label{G-2}
G^{(2)}=\{(a,b)\in G \times G~|~L^{\rhd}(a)= \Ad_{b}\},
\end{eqnarray}
$i$ is the inclusion, $i_2$ is the inclusion to the second summand and $p_1,p_2$ are projections to the first summand and the second summand respectively. According to Eq. \eqref{G-2}, it is obvious that the right square commutes.

Now choose a section of the central extension of groups in the first line of the diagram \eqref{diagram2}, which is a map $S:(G,\circ) \to G^{(2)}$ such that $p_1\circ S=\Id_{G}$. It is obvious that $S$ is of  the form $S(a)=(a,\Phi(a))$ for a map $\Phi:G \to G$. Since the right square commutes, we have
$$ 
\Ad \circ p_2 \circ S(a)=L^{\rhd} \circ p_1 \circ S(a)=L^{\rhd}(a), \quad \forall a \in G,
$$
which implies that $L^{\rhd}(a)=\Ad_{\Phi(a)}$ for all $a\in G$.

Define $\omega:G \times G \to G^{(2)}$ by
$$
\omega(a,b)=S(b)^{-1} \bullet S(a)^{-1} \bullet S(a\circ b),\quad \forall a,b\in G.
$$
Obviously, $p_1(\omega(a,b))=e$, which implies that $p_1(\omega(a,b))\in i_2(Z(G)).$ Identifying $Z(\g)$ and $i_2(Z(\g))$, it follows that $\omega:G \times G \to Z(G)$ is given  by
\begin{eqnarray}\label{omega}
\omega(a,b)=\Phi(b)^{-1} \cdot \Phi(a)^{-1} \cdot \Phi(a \circ b),\quad \forall a,b\in G.
\end{eqnarray}

On the other hand, the induced group action $\mu_{S}:(G,\circ)\to \Aut(Z(G))$ of the section $S$ is given by
$$  
\mu_{S}(a)(b)=\Ad^{\bullet}_{(a,\Phi(a))} (e,b)=b,\quad \forall a\in G,b\in Z(G), $$
which is trivial.

\begin{pro}\label{prop:InnerGroup2cocycle}
Let $(G,\cdot,\rhd)$ be an inner post-group, $S$ and $\omega$ defined as above. Then $\omega$ is a~2-cocycle on the group $(G,\circ)$ with the coefficients in the trivial group action $Z(G)$. Moreover, its cohomology class in $H^2((G,\circ), Z(G))$ does not depend on the choice of the section $S$.
\end{pro}

\begin{proof}
In fact, this conclusion follows from the general theory of extensions of groups. To be self-contained, we roughly give the proof.

By the associative law of the group $(G,\cdot)$ and Eq. \eqref{omega}, for all $a,b,c\in G$, we have
\begin{eqnarray*}
&& \omega(b,c)\cdot \omega(a,b \circ c)\\
&=& \Phi(c)^{-1} \cdot \Phi(b)^{-1} \cdot \Phi(b \circ c)\cdot \Phi(b \circ c)^{-1} \cdot \Phi(a)^{-1}
\cdot \Phi(a\circ b \circ c)\\
&=& \Phi(b)^{-1} \cdot \Phi(a)^{-1} \cdot \Phi(a) \cdot \Phi(b) \cdot \Phi(c)^{-1} \cdot \Phi(b)^{-1}\cdot \Phi(a)^{-1} \cdot \Phi(a\circ b \circ c)\\
&=& \Phi(b)^{-1} \cdot \Phi(a)^{-1} \cdot (a \rhd (b \rhd \Phi(c)^{-1})) \cdot \Phi(a\circ b \circ c)\\
&=& \Phi(b)^{-1} \cdot \Phi(a)^{-1} \cdot ((a \circ b)\rhd \Phi(c)^{-1}) \cdot \Phi(a\circ b \circ c)\\
&=& \Phi(b)^{-1} \cdot \Phi(a)^{-1} \cdot \Phi(a \circ b) \cdot  \Phi(c)^{-1} \cdot \Phi(a \circ b)^{-1} \cdot \Phi(a\circ b \circ c)\\
&=& \omega(a,b)\cdot \omega(a\circ b,c),
\end{eqnarray*}
which implies that $\omega$ is a $2$-cocycle on the group $(G,\circ)$ with the coefficients in the trivial group action $Z(G)$.

Suppose that there exists another section $S':(G,\circ) \to G^{(2)}$ given by $S'(a)=(a,\Phi'(a))$ for a~map $\Phi':G \to G$. Since the right square commutes, we have
$$\Ad_{\Phi'(a)} =\Ad \circ p_2 \circ S'(a)=L^{\rhd}\circ p_1 \circ S'(a)
=L^{\rhd}(a),\quad \forall a,b\in G. $$
Define $\zeta:G \to G$ by $\zeta(a):=(\Phi'(a))^{-1}\cdot \Phi(a)$, then we have
$$ a\rhd b=\Phi(a) \cdot b \cdot (\Phi(a))^{-1}=\Phi'(a) \cdot b \cdot (\Phi'(a))^{-1},  $$
which implies that $b=\zeta(a) \cdot b \cdot (\zeta(a))^{-1}$.
That means, $\Imm(\zeta) \subseteq Z(G).$

Next we define $\omega':G \times G \to Z(G)$ by
\begin{eqnarray*}
\Phi'(a \circ b)=\Phi'(a) \cdot \Phi'(b) \cdot \omega'(a,b),\quad \forall a,b\in G.
\end{eqnarray*}
Then we have
\begin{eqnarray*}
&&\zeta(a)\cdot \zeta(b) \cdot \omega(a,b) \cdot \zeta(a \circ b)^{-1}\\
&=& (\Phi'(a))^{-1}\cdot \Phi(a) \cdot (\Phi'(b))^{-1}\cdot \Phi(b) \cdot  \Phi(b)^{-1} \cdot \Phi(a)^{-1} \cdot \Phi(a \circ b) \cdot (\Phi(a \circ b))^{-1} \cdot \Phi'(a \circ b)\\
&=& (\Phi'(a))^{-1}\cdot \Phi(a) \cdot (\Phi'(b))^{-1} \cdot \Phi(a)^{-1} \cdot \Phi'(a \circ b)\\
&=& (\Phi'(a))^{-1}\cdot \Phi(a) \cdot (\Phi'(b))^{-1} \cdot \Phi(a)^{-1} \cdot \Phi'(a) \cdot (\Phi'(a))^{-1} \cdot \Phi'(a \circ b)\\
&=& \zeta(a) \cdot (\Phi'(b))^{-1} \cdot  (\zeta(a))^{-1} \cdot (\Phi'(a))^{-1}
\cdot \Phi'(a \circ b)\\
&=& (\Phi'(b))^{-1} \cdot (\Phi'(a))^{-1} \cdot \Phi'(a \circ b)\\
&=& \omega'(a,b),
\end{eqnarray*}
which implies that $[\omega']=[\omega]$ in the second cohomology group $H^2((G,\circ), Z(G))$. Thus, $[\omega]$ does not depend on the choice of $S$.
\end{proof}

\begin{defi}
The cohomological class $[\omega]\in H^2((G,\circ), Z(G))$ is called the {\bf obstruction class} of the inner post-group being induced by a Rota-Baxter operator.
\end{defi}

\begin{thm}\label{omega-RB-operator}
An inner post-group $(G,\cdot,\rhd)$ is induced by a Rota-Baxter operator of weight~1 on the group $(G,\cdot)$ if and only if the obstruction class $[\omega]$ in $H^2((G,\circ), Z(G))$ is trivial.
\end{thm}

\begin{proof}
If $\omega(a,b)=e$ (here e is the unit of the group $(G,\cdot)$), for all $a,b\in G$, then we have
$$
\Phi(a)\cdot \Phi(b)=\Phi(a \circ b)= \Phi(a \cdot \Ad_{\Phi(a)} b),
$$
which implies that $\Phi$ is a Rota-Baxter operator of weight $1$ on the group $(G,\cdot)$.

Suppose that the obstruction class $[\omega]$ in $H^2((G,\circ), Z(G))$ is trivial, i.\,e. $[\omega]=[e]$.
That means, there exists a map $\zeta:G \to Z(G)$ such that
\begin{eqnarray}\label{omega-zeta}
\omega(a,b)=\zeta(a)\cdot \zeta(b) \cdot \zeta(a\circ b)^{-1}, \quad \forall a,b\in G.
\end{eqnarray}
Define $\huaB: G \to G$ by $\huaB(a)=\Phi(a)\zeta(a)$ for all $a\in G$. Let us show that $\huaB$ is
a Rota-Baxter operator of weight $1$ on the group $(G,\cdot)$:
\begin{eqnarray*}
\huaB(a \circ b)&=&\Phi(a \circ b)\zeta(a \circ b)\\
&\overset{\eqref{omega},\,\eqref{omega-zeta}}{=}&  \Phi(a)\cdot \Phi(b) \cdot \omega(a,b) \cdot \omega(a,b)^{-1} \cdot \zeta(a) \cdot \zeta(b)\\
&=& \Phi(a)\cdot \zeta(a) \cdot \Phi(b)\cdot \zeta(b)\\
&=& \huaB(a) \cdot \huaB(b),
\end{eqnarray*}
since $\Imm(\zeta) \subseteq Z(G)$. The other direction of the proof is obvious, so we omit details.
\end{proof}

Recall from \cite{BGST} how to obtain a post-Lie algebra from a post-Lie group via differentiation.

\begin{defi}{\rm (\cite{BGST})}
A post-group $(G,\cdot,\rhd)$ is called a {\bf post-Lie group} if $(G,\cdot)$ is a Lie group and $\rhd$ is a smooth map.
\end{defi}	

Let $(\g,[\cdot,\cdot]_\g)$ be the Lie algebra of the Lie group $(G,\cdot)$. Denote by $\Aut(\g)$ and $\Der(\g)$ the Lie group of automorphisms and the Lie algebra of derivations on the Lie algebra $(\g,[\cdot,\cdot]_\g)$ respectively.  Let $\exp:\g\to G$ denote the exponential map. Then the relation between the Lie bracket $[\cdot,\cdot]_\g$ and the Lie group multiplication is given by the following fundamental formula:
\begin{equation}\label{eq:expo}
[u,v]_\g=\frac{d^2}{dt\,ds}\,\bigg|_{t,s=0}\exp(tu)\exp(sv)\exp(-tu),\quad \forall u,v\in\g.
\end{equation}

Since $L^\rhd(a)\in\Aut(G)$, it follows that $(L^\rhd(a))_{*e}\in \Aut(\g)$. Thus we obtain a map, still denoted by $L^\rhd$, from $G$ to $\Aut(\g)$. Then taking the differentiation, we obtain a map $L^\rhd_{*e}: \g\to \Der(\g)$. The above process can be summarized by the following diagram:
\begin{equation}\label{eq:relation}
\begin{split}
\small{
\xymatrix{G \ar[rr]^{L^\rhd}\ar[d]_{\text{differentiation}} &  &  \Aut(\g) \ar[d]^{\text{differentiation}}  \\
\g \ar[rr]^{ L^\rhd_{*e} } & &\Der(\g).}
}
\end{split}
\end{equation}
Define $\triangleright:\g\otimes \g\to \g$ by
\begin{equation}\label{eq:diff}
x\triangleright y= L^\rhd_{*e} (x)(y)=\frac{d}{dt}\bigg|_{t=0} L^\rhd(\exp(tx))(y)=\frac{d}{dt}\bigg|_{t=0}\frac{d}{ds}\bigg|_{s=0}L^\rhd(\exp(tx))\exp(sy).
\end{equation}

\begin{thm}{\rm (\cite{BGST})}\label{thm:diffpL}
Let $(G,\cdot,\rhd)$ be a post-Lie group with the smooth multiplication $\rhd$. With above notations, then  $(\g,[\cdot,\cdot]_\g,\triangleright)$ is a post-Lie algebra.
\end{thm}

Now we show that an inner post-Lie group gives rise to an inner post-Lie algebra via differentiation.

\begin{pro}\label{pro:obs-post-group}
Let $(G,\cdot,\rhd)$ be an inner post-Lie group. Then its differentiation
$(\g,[\cdot,\cdot]_\g,\triangleright)$ is an inner post-Lie algebra. Moreover, if the obstruction class $[\omega]\in H^2((G,\circ), Z(G))$ is trivial, then the obstruction class $[\kappa] \in H^2(\g_\triangleright, Z(\g))$ is also trivial.
\end{pro}

\begin{proof}
If $(G,\cdot,\rhd)$ is an inner post-Lie group, that means, $\Imm(L^{\rhd})\subseteq \Inn(G),$ then by diagram~\eqref{eq:relation} and Eq.~\eqref{eq:diff}, it is straightforward to check that $\Imm(L_{\triangleright}) \subseteq \Inn(\g),$ which implies that the differentiation $(\g,[\cdot,\cdot]_\g,\triangleright)$ is an inner post-Lie algebra.

If the obstruction class $[\omega]\in H^2((G,\circ), Z(G))$ is trivial, by Theorem \ref{omega-RB-operator}, there is a Rota-Baxter operator $\huaB:G \to G$ such that $a \rhd b =\Ad_{\huaB(a)}b=\huaB(a) \cdot b \cdot (\huaB(a))^{-1}$ for all $a,b\in G$. By \cite[Theorem 2.10]{GLS}, $B:=\huaB_{\ast e}: \g \to \g$ is a Rota-Baxter operator on $\g$ and we have
\begin{eqnarray*}x \triangleright y&=& \frac{d}{dt}\bigg|_{t=0}\frac{d}{ds}\bigg|_{s=0}L^\rhd(\exp(tx))\exp(sy)\\
&=&
\frac{d}{dt}\bigg|_{t=0}\frac{d}{ds}\bigg|_{s=0}  \Ad_{\huaB(\exp(tx))}  \exp(sy)\\
&=& \ad_{B(x)} y=[B(x),y]_{\g},
\end{eqnarray*}
for all $x,y \in \g$, which implies that $(\g,[\cdot,\cdot]_\g,\triangleright)$ is an inner post-Lie algebra induced by a Rota-Baxter operator $B$ on $\g$. By Theorem \ref{kap-RB-operator}, the corresponding obstruction class $[\kappa] \in H^2(\g_\triangleright, Z(\g))$ is trivial.
\end{proof}

\section{Applications of inner post-Lie algebras and inner post-Lie groups}

In this section, we give the applications of inner post-Lie algebras and inner post-Lie groups. We first consider inner post-Lie algebras and recall a result about semisimple Lie algebras.

\begin{lem}{\rm (\cite{Burde2022})} \label{lem1}
Let $\g$ be a semisimple Lie algebra decomposed as a sum $\g = \g_1 + \g_2$
of its two semisimple subalgebras $\g_1,\g_2$. Then the subalgebra $\g_1 \cap \g_2$ is either zero or semisimple.
\end{lem}

\begin{thm}\label{thm-2}
Let $(\g,[\cdot,\cdot]_\g, \rk)$ be a finite dimensional inner post-Lie algebra such that the obstruction class $[\kappa]$ in $H^2(\g_\triangleright, Z(\g))$ is trivial. If the sub-adjacent Lie algebra $\g_{\rk}$ is semisimple, then $\g\cong \g_{\rk}$.
\end{thm}

\begin{proof}	
By Theorem \ref{kap-RB-operator}, there is a \rb operator $R:\g\to\g$ such that $x\rk y=[R(x),y]_\g$ for all $x,y\in \g$. If $R$ is trivial (i.e. $R=0$ or $R=-\id$), then $\g\cong \g_{\rk}$.

Now suppose that $R$ is not trivial. Note that $R$ is also a \rb operator of weight 1 on the Lie algebra $\g_{\rk}$.
Naturally, we obtain   a new post-Lie structure by Lemma \ref{lem-RB-innPL}.
Repeating this process,  we obtain a series of Rota-Baxter Lie algebras  $\{\g_i\}_{i\geq 0}$ on the same vector space $\g$ which are defined by
\begin{align*}
[x,y]_0 &= [x,y]_\g,\\
[x,y]_{i+1} &= [R(x),y]_i + [x,R(y)]_i+[x,y]_i, \quad \forall x,y\in \g,
\end{align*}
and a series of inner post-Lie algebras $\{(\g_i,[\cdot,\cdot]_i,\rk_i)\}$ which is given by $x\rk_i y=[R(x),y]_i$ for all $\quad i\geq 0$.
We observe that $R,\, R+\id:\mathfrak{g}_{i+1} \to \mathfrak{g}_i$ are Lie algebra homomorphisms for $i\geq 0$.
Hence we obtain a composition of homomorphisms
\[
\mathfrak{g}_i \xrightarrow[R+\id]{R} \mathfrak{g}_{i-1}
\xrightarrow[R+\id]{R} \cdots  \xrightarrow[R+\id]{R} \mathfrak{g}_{0}.
\]
So $\ker(R^k)$ and $\ker((R+\id)^k)$ are
ideals in $\mathfrak{g}_j$ for all $1\leq k\leq j$.

Next we show that $\g_2$ is also a semisimple Lie algebra. By \cite[Theorem 4.6]{Burde2020-1}, if $\g_1$ is semisimple, then $\g_0$ is semisimple too.  Fix $R,R+\id: \g_1\to \g_0$. We have a decomposition
\begin{align*}\label{Sum}
\g_0= \Imm(R) + \Imm(R+\id),
\end{align*}
since $x = R(-x) + (R+\id)(x)$ for every $x\in \g_0$.  Hence, both $\Imm(R)$ and $\Imm(R+\id)$
are semisimple as nonzero homomorphic images of the semisimple Lie algebra~$\mathfrak{g}_1$. Then by Lemma \ref{lem1}, $\Imm(R) \cap \Imm(R+\id)$ is either zero or semisimple. We claim that $R(\mathfrak{g}_1)\cap (R+\id)(\mathfrak{g}_1)= R(R+\id)(\mathfrak{g}_2)$. On the one hand,
for all $x\in R(\mathfrak{g}_1)\cap (R+\id)(\mathfrak{g}_1)$, we choose $x_1,x_2\in \g$ such that $x=R(x_1)=(R+\id)(x_2)$. Hence, $R(x_1-x_2)=x_2$ which means that $$x=(R+\id)R(x_1-x_2)=R(R+\id)(x_1-x_2)\in \Imm R(R+\id).$$ Therefore, $R(\mathfrak{g}_1)\cap (R+\id)(\mathfrak{g}_1)\subseteq R(R+\id)(\mathfrak{g}_2)$. On the other hand, since $R(R+\id)=(R+\id)R$, we have $R(\mathfrak{g}_1)\cap (R+\id)(\mathfrak{g}_1)\supseteq R(R+\id)(\mathfrak{g}_2)$. Therefore, $R(R+\id)(\mathfrak{g}_2)$ is semisimple or zero.

Note that both $\ker(R)$ and $\ker(R+\id)$ have the same product in all $\mathfrak{g}_i$ respectively and they are not intersecting ideals in $\mathfrak{g}_j$ for $j\geq1$.
So, $\ker(R)\oplus\ker(R+\id)$ as an ideal in the semisimple algebra $\mathfrak{g}_1$ is itself semisimple.
Then $\ker(R)\oplus\ker(R+\id)$ is a~semisimple ideal in $\mathfrak{g}_2$.
Note that $\ker R(R+\id)=\ker(R)\oplus\ker(R+\id)$ by inclusion relationship.
Therefore, $\ker(R(R+\id))$ is a~semisimple ideal in $\mathfrak{g}_2$.
If $\ker(R(R+\id)) = \mathfrak{g}_2$, then $\mathfrak{g}_2$ is semisimple. Suppose that $\ker(R(R+\id)) \neq \mathfrak{g}_2$.
In summary, for the homomorphism $R(R+\id)\colon \mathfrak{g}_2\to\g_0$, both its kernel and image are semisimple. It means that $\mathfrak{g}_2$ is semisimple since $\g_2 \big/ \ker R(R+\id)\cong \Imm R(R+\id)$. Continuing on by induction, we obtain $\mathfrak{g}_n$ is semisimple, where $n = \dim(V)$. Then by \cite[Corollary 4.11]{Burde2020-1}, all $\mathfrak{g}_i$ are pairwise isomorphic. The proof is finished.
\end{proof}

Now we consider inner post-Lie groups and obtain a similar result. Here we assume that all Lie groups are simply connected, in order to establish the relationship with their corresponding Lie algebras.

\begin{lem}\label{lem-group-1}
Let $G$ be a simply connected semisimple Lie group. Suppose that $G=G_1\cdot G_2$ as a product of its two connected semisimple Lie subgroups $G_1$ and $G_2$. Then the Lie subgroup $G_1 \cap G_2$ is either a discrete subgroup or a closed semisimple Lie subgroup. More precisely, the identity component of  $G_1 \cap G_2$ is trivial or a connected semisimple Lie group.
\end{lem}

\begin{proof}
Let $\g,\g_1,\g_2$ be the corresponding Lie algebras of Lie groups $G,G_1,G_2$. Since $G$ is simply connected and semisimple, $\g$ is semisimple. Since $G_1,G_2$ are connected and semisimple, $\g_1,\g_2$ are semisimple as Lie subalgebras. Denote by $H=G_1 \cap G_2$, $H$ is a closed subgroup of $G$, whose corresponding Lie algebra is exactly $\g_1 \cap \g_2$. By Lemma \ref{lem1}, if $\g_1 \cap \g_2=0$, then $H$ is a discrete subgroup. Otherwise, $H$ is a connected semisimple Lie subgroup.
\end{proof}

\begin{lem}\label{lem-group-2}
Let $(G,\cdot,\rhd)$ be an inner post-Lie group such that the obstruction class $[\omega]\in H^2((G,\circ), Z(G))$ is trivial. If the sub-adjacent Lie group $(G,\circ)$ is simply connected and semisimple, then $(G,\cdot)$ is also semisimple.
\end{lem}

\begin{proof}
If $G$ is a simply connected Lie group, $G$ is semisimple if and only if $H_{\bf smooth}^1(G,M)=0$ for any finite dimensional smooth $G$-module $M$.
It is the first Whitehead Lemma applied to the version of Lie groups.
Since $(G,\circ)$ is simply connected and semisimple, we have $H_{\bf smooth}^1((G,\circ),M)=0$ for any finite dimensional smooth $G$-module $M$. By Proposition \ref{pro:obs-post-group}, since $(G,\cdot,\rhd)$ is an inner post-Lie group such that the obstruction class $[\omega]\in H^2((G,\circ), Z(G))$ is trivial, its differentiation is the inner post-Lie algebra $(\g,[\cdot,\cdot]_\g,\triangleright)$ such that the obstruction class $[\kappa] \in H^2(\g_\triangleright, Z(\g))$ is trivial.
Denote by $\g_0,\g_1$ the Lie algebras $(\g,[\cdot,\cdot]_\g)$ and $(\g_\rk,[\cdot,\cdot]_{\triangleright})$ respectively.
Since $(G,\circ)$ and $(G,\cdot)$ are simply connected, we have $H_{\bf smooth}^1((G,\cdot),M)=H^1(\g_0,M)$ and $H_{\bf smooth}^1((G,\circ),M)=H^1(\g_1,M)=0$.
According to the proof of \cite[Theorem 4.6]{Burde2020-1}, we have $H^1(\g_0,M)=0$ from $H^1(\g_1,M)=0$, which implies that $H_{\bf smooth}^1((G,\cdot),M)=0$. That means, if $(G,\circ)$ is semisimple, then $(G,\cdot)$ is also semisimple.
\end{proof}

Let $(G,\cdot,\rhd)$ be an inner post-Lie group such that the obstruction class $[\omega]\in H^2((G,\circ), Z(G))$ is trivial. By Theorem \ref{omega-RB-operator}, there is a Rota-Baxter operator $\huaB:G \to G$ such that
$a \rhd b =\Ad_{\huaB(a)}b=\huaB(a) \cdot b \cdot (\huaB(a))^{-1}$ for all $a,b\in G$.

If $\huaB$ is trivial, i.\,e. $\huaB(a)=e$ or $\huaB(a)=a^{-1}$ for all $a\in G$, it is easy to check that $(G,\circ)= (G,\cdot)$ or $(G,\circ)= (G,\cdot^{\bf op})$.

Now suppose that $\huaB$ is nontrivial. Note that $\huaB$ is also a Rota-Baxter operator of weight $1$ on the sub-adjacent Lie group $(G,\circ)$ by \cite[Proposition 2.14]{GLS}. Thus, we obtain a Rota-Baxter Lie group $(G,\circ,R)$, which can induce a new post-Lie group by Lemma \ref{lem:RB-group-post-group}. Repeating this process, we obtain a series of Rota-Baxter Lie groups $\{(G_i,\circ_i)\}_{i\geq 0}$ on the same manifold $G$ defined by
\begin{eqnarray*}
a \circ_0 b&:=& a \cdot b;\\
a \circ_{i+1} b&:=& a \circ_i (\Ad_{\huaB(a)}^{\circ_i} b) = a \circ_i (\huaB(a)\circ_i  b \circ_i \overline{\huaB(a)}^{i}), \quad \forall a,b\in G,
\end{eqnarray*}
where $\overline{\huaB(a)}^{i}$ is the inverse of $\huaB(a)$ in the group $(G_i,\circ_i)$. By Lemma \ref{lem:RB-group-post-group}, we also have a series of inner post-Lie groups $\{(G_i,\circ_i,\rhd_i)\}_{i\geq 0}$, where $\rhd_i $ is given by $a \rhd_i b=\Ad_{\huaB(a)}^{\circ_i} b=\huaB(a)\circ_i  b \circ_i \overline{\huaB(a)}^{i}$ for all $a,b\in G$. According to \cite[Proposition 3.1]{GLS}, by induction, it is straightforward to check that $\huaB,\widetilde{\huaB}:(G_i,\circ_i) \to  (G_{i-1},\circ_{i-1})$ are homomorphisms of Lie groups for all $i\geq 1$, where $\widetilde{\huaB}:(G_i,\circ_i) \to  (G_{i-1},\circ_{i-1})$ is given by $\widetilde{\huaB} (a)=a \cdot \huaB(a)$ for all $a\in G$.
Hence we obtain a~composition of homomorphisms of Lie groups:
\[
G_i \xrightarrow[\widetilde{\huaB}]{\huaB} G_{i-1}
\xrightarrow[\widetilde{\huaB}]{\huaB} \cdots  \xrightarrow[\widetilde{\huaB}]{\huaB} G_{0}.
\]

\begin{thm}\label{thm-3}
With above notations, let $(G,\cdot,\rhd)$ be an inner post-Lie group, where $\rhd$ is induced by the above non-trivial Rota-Baxter operator $\huaB$ of weight $1$. Assume that $\huaB$ and $\widetilde{\huaB}$, as homomorphisms of Lie groups, have differentials $\dM\huaB$ and $\dM\widetilde{\huaB}$ that are not injective.
If the sub-adjacent Lie group $(G,\circ)$ is simply connected and semisimple, then all Lie groups $(G_n,\circ_n)$ are semisimple for all $n\geq 0$.
\end{thm}

\begin{proof}
By Lemma \ref{lem-group-2}, if $(G,\circ)$ is simply connected and semisimple, then $(G,\cdot)$ is also semisimple. Fix $\huaB,\widetilde{\huaB}: (G,\circ) \to (G,\cdot)$. We have a decomposition
$$
G_0=(G,\cdot) \cong \Imm(\widetilde{\huaB}) \cdot \Imm(\huaB),
$$
since we have $a=a \cdot \huaB(a) \cdot \huaB(a)^{-1}\in \Imm(\widetilde{\huaB}) \cdot \Imm(\huaB)$ for all $a \in G$. Hence, both $\Imm(\widetilde{\huaB})$ and $\Imm(\huaB)$ are connected semisimple Lie groups as nonzero homomorphic images of the semisimple Lie group $(G,\circ)$. Then by Lemma \ref{lem-group-1}, $\Imm(\widetilde{\huaB}) \cap \Imm(\huaB)$ is either a discrete subgroup or a closed semisimple Lie subgroup. However, by \cite[Lemma 1]{GG}, there holds $\widetilde{\huaB}(G_1) \cap \huaB(G_1)=\huaB \widetilde{\huaB} (G_2)$. Therefore, $\widetilde{\huaB}(G_1) \cap \huaB(G_1)$ is not a discrete subgroup. Consequently, $\huaB \widetilde{\huaB} (G_2)$ is semisimple.

Note that $\ker(\huaB)$ and $\ker(\widetilde{\huaB})$ have the same product in all $G_i$ respectively.
It is easy to check that
$\ker(\huaB \widetilde{\huaB})=\ker(\huaB) \times \ker(\widetilde{\huaB})$
as a direct product of $\ker(\huaB)$ and $\ker(\widetilde{\huaB})$ because of
$\huaB \widetilde{\huaB} = \widetilde{\huaB} \huaB$ (\cite[Lemma~5c]{RBGroups}).
Since $\huaB$ and $\widetilde{\huaB}$, as homomorphisms of Lie groups, have differentials $\dM\huaB$ and $\dM\widetilde{\huaB}$ that are not injective,  $\ker(\huaB)$ and $\ker(\widetilde{\huaB})$ are not discrete subgroups.
Thus, $\ker(\huaB)$ and $\ker(\widetilde{\huaB})$ are connected closed semigroup Lie subgroups, which imply that $\ker(\huaB \widetilde{\huaB})$ is also semisimple.

In conclusion, for the Lie group homomorphism $\huaB \widetilde{\huaB}: (G_2,\circ_2) \to (G_0,\cdot)$, both its kernel and image are semisimple.
Since $G_2/\ker(\huaB \widetilde{\huaB}) \cong \Imm(\huaB \widetilde{\huaB})$,
we obtain that $(G_2,\circ_2)$ is a semisimple Lie group.
Continuing on by induction, we obtain that all Lie groups $(G_n,\circ_n)$ are semisimple for all $n\geq 0$.
The proof is finished.
\end{proof}

\end{document}